\documentclass[11pt]{article}

\usepackage{geometry}
\usepackage{rotating}

\usepackage{kbordermatrix}
\usepackage{exscale}
\usepackage{tabularx}
\usepackage{latexsym}
\usepackage{amsmath,amssymb,amstext,amsthm} 
\usepackage{graphicx,color,epsfig}
\usepackage{graphicx}
\usepackage{float}
\usepackage{verbatim}
\usepackage{algorithm2e}
\usepackage[bookmarks,pagebackref,
    pdfpagelabels=true, 
    ]{hyperref}
\usepackage{seceqn}
\usepackage{lineno}
\usepackage{makeidx}
\makeindex

\usepackage{blkarray}

\usepackage{amscd}
\usepackage{url}

\usepackage{float}
\restylefloat{table}
\DeclareMathAlphabet{\mathpzc}{OT1}{pzc}{m}{it}  

\newcommand{\GIF}{\mbox{\texttt{GIF}}}

\newtheorem{theorem}{Theorem}[section]   

\def\R{\mathbb{R}}

\def\eqref#1{{\normalfont(\ref{#1})}}

\definecolor{rblue}{rgb}{.255,.41,.884} 

\def\eqref#1{{\normalfont(\ref{#1})}}

\newtheorem{defi}{Definition}[section]
\newtheorem{example}{Example}[section]

\newtheorem{prob}{Problem}[section]

\newtheorem{remark}{Remark}[section]

\newtheorem{lemma}{Lemma}[section]

\newtheorem{defin}{Definition}[section]

\newcommand{\textdef}[1]{\textit{#1}\index{#1}}
\DeclareMathOperator{\nul}{{null}}

\newcommand{\Ss}{{\mathcal S}}

\newcommand{\PP}{{\mathcal P} }

\newcommand{\QED}{{\flushright{\hfill ~\rule[-1pt] {8pt}{8pt}\par\medskip ~~}}}

\newcommand{\Snp}{{\mathcal S^n_+}}

\newcommand{\Sk}{{\mathcal S^{k}}\,}
\newcommand{\Skp}{{\mathcal S^{k}_+}\,}

\newcommand{\A}{{\mathcal A}}

\newcommand{\F}{{\mathcal F\,}}

\newcommand{\bbm}{\begin{bmatrix}}
\newcommand{\ebm}{\end{bmatrix}}
\newcommand{\bem}{\begin{pmatrix}}
\newcommand{\eem}{\end{pmatrix}}
\newcommand{\beq}{\begin{equation}}
\newcommand{\beqs}{\begin{equation*}}
\newcommand{\bet}{\begin{table}}
\newcommand{\eeq}{\end{equation}}
\newcommand{\eeqs}{\end{equation*}}
\newcommand{\beqr}{\begin{eqnarray}}

\newcommand{\Vecc}{\mbox{vec}}

\DeclareMathOperator{\trace}{{trace}}

\DeclareMathOperator{\svec}{{s2vec}}

\DeclareMathOperator{\rank}{{rank}}
\DeclareMathOperator{\spanl}{{span}}

\newcommand{\nc}{\newcommand}
\nc{\arrow}{{\rm arrow\,}}
\nc{\Arrow}{{\rm Arrow\,}}
\nc{\BoDiag}{{\rm B^0Diag\,}}
\nc{\bodiag}{{\rm b^0diag\,}}

\nc{\Mm}{{\mathcal M}^{m} }
\nc{\Mmn}{{\mathcal M}^{mn} }
\nc{\Mnr}{{\mathcal M}_{nr} }
\nc{\Mnmr}{{\mathcal M}_{(n-1)r} }
\nc{\kwqqp}{Q{$^2$}P\,}
\nc{\kwqqps}{Q{$^2$}Ps}
\def\argmin{\mathop{\rm argmin}}

\nc{\notinaho}{(X,S)\in \overline{AHO}(\A)}
\nc{\inaho}{(X,S)\in AHO(\A)}

\newcommand{\bea}{\begin{eqnarray}}%
\newcommand{\eea}{\end{eqnarray}}%
\newcommand{\beas}{\begin{eqnarray*}}%
\newcommand{\eeas}{\end{eqnarray*}}%
\newcommand{\Int}{{\rm int\,}}

\renewcommand{\F}{\mathcal{F}}%
{}

\newcommand{\Hnp}[1][]{\,\mathbb{H}_+^{\ifthenelse{\equal{#1}{}}{n}{#1}}}
\newcommand{\Hn}[1][]{\,\mathbb{H}^{\ifthenelse{\equal{#1}{}}{n}{#1}}}
\newcommand{\Dn}[1][]{\,\mathbb{D}^{\ifthenelse{\equal{#1}{}}{n}{#1}}}

\begin{document}

\RestyleAlgo{boxruled}
\bibliographystyle{plain}
\title{Computing the generators of the truncated real radical ideal by moment matrices and SDP facial reduction
}

\author{
\href{http://www.apmaths.uwo.ca/people/grad_students.html}
{Fei Wang}
        \thanks{
Dept. Appl. Math., University of Western Ontario, London, Ontario,
Canada}
\and
\href{http://www.orcca.on.ca/~reid/NewWeb/}{Greg Reid}
        \thanks{
Dept. Appl. Math., University of Western Ontario, London, Ontario,
Canada}
\and
\href{http://orion.math.uwaterloo.ca/~hwolkowi/}{Henry Wolkowicz}%
        \thanks{
Department of Combinatorics and Optimization,
          University of Waterloo, Waterloo, Ontario N2L 3G1, Canada.
Research supported in part by The Natural Sciences and Engineering
                Research Council of Canada (NSERC).
}
}
\maketitle

\tableofcontents

\begin{abstract}
Recent breakthroughs have been made in the use of semidefinite programming and its application to real polynomial
solving.  For example, the real radical of a zero dimensional ideal, can be determined by such approaches as shown by Lasserre and collaborators.
Some progress has been made on the determination of the real radical in positive dimension by Ma, Wang and Zhi.
Such work involves the determination of maximal rank semidefinite moment matrices.  Existing methods are computationally
expensive and have poorer accuracy on larger examples.

This paper is motivated by problems in the numerical computation of the real radical ideal in the general positive case. 

In this paper we give a method to compute the generators of the real radical for any given degree $d$. We combine the use 
of moment matrices and techniques from SDP optimization: facial reduction first developed by Borwein and Wolkowicz. In use of
the semidefinite moment matrices to compute the real radical, the maximum rank property is very key, 
and with facial reduction, it can be guaranteed with very high accuracy. Our algorithm can be used to test the real radical membership of a given polynomial. In a special situation,  we can determine the real radical ideal in the positive dimensional case. 
\end{abstract}
\section{Introduction}
The breakthrough work of Lasserre and collaborators
\cite{LasserreLaurentRostalski09,MR2830310}
shows that the \textdef{real radical ideal, RRI}, of a real polynomial system
with finitely many solutions can be
\index{RRI, real radical ideal}
determined by computing the kernel of so-called moment matrices arising from a \textdef{semidefinite programming (SDP)}
feasibility problem\index{SDP, semidefinite programming}. This RRI is generated by a system of real polynomials
having only real roots that are free of multiplicities.
The number of such real roots may be considerably less than
the number of complex roots (see the paper \cite{reidwangwolk2015} for examples and references).
Global numerical solvers, such as homotopy continuation solvers typically compute all real roots by first computing all complex
(including real) roots. And if the roots have multiplicity,
then elaborate strategies are needed to avoid difficulties that arise
as the paths from the homotopy solvers
approach these singular roots \cite{SommeseWampler05}.
A conjectured extension of such methods to positive dimensional polynomial
systems has been given recently by Ma, Wang and Zhi \cite{MWZ:2012,MA12}.

Our approach also builds on the method of moment matrices. A key step is to solve the problem of the following type for $X$
\begin{equation}
\label{eq:sdpfeas}
  \A (X) = b, \quad X \in \Skp, X \;\mbox{is maximum rank},
\end{equation}
where $\Skp$ denotes the convex cone of
\index{$\Skp$, semidefinite cone}
\index{semidefinite cone, $\Skp$}
$k\times k$ real symmetric positive semidefinite matrices, and
$\A : \Skp \rightarrow \R^l$ is a linear transformation which enforces the moment matrix structure for $X$.

The standard regularity assumption for \eqref{eq:sdpfeas}
is the \textdef{Slater constraint qualification} or strict feasibility
assumption:
\begin{equation}
\label{eq:Slater}
\text{there exists }  X \text{ with }
  \A  X = b, \quad  X \in \Int \Skp.
\end{equation}
We let $X\succeq 0, \succ 0$ denote $X \in \Skp, \in \Int \Skp$,
respectively.
It is well known that the Slater condition for SDP holds generically,
e.g.,~\cite{DurJaSt:12}. Surprisingly,
many SDP problems arising from particular applications,
and in particular our polynomial system applications,
are marginally infeasible, i.e.,~fail to satisfy strict feasibility.
This means that the feasible set lies within the boundary of the cone,
which creates difficulties with numerical algorithms such as interior point solvers and the maximum rank can not be 
computed accurately.
To help regularize such SDP problems,
facial reduction was introduced
in 1982 by Borwein and Wolkowicz \cite{bw1,bw3}.
However it was only much later that the power of facial reduction was
exhibited in many applications, e.g.,~\cite{KaReWoZh:94,WoZh:96,AlWo:99}.
Developing algorithmic implementations of facial reduction that work
for large classes of SDP problems and the connections with
perturbation and convergence analysis has recently been achieved in
e.g.,~\cite{kriswolk:09,ChDrWo:14,ChWosensit:14,DrusLiWolk:14}.

In this paper,  we use facial reduction approach to effectively reduce the size of the SDP problem associated with the 
input polynomial system so that 
it is strictly feasible and then solve the reduced problem
 using the Douglas-Rachford reflection method. We then use the \emph{geometric involutive basis} to check if 
 the kernel of the moment matrix is a truncated ideal (ideal-like). This leads to a method to compute the generators of real radicals up to 
any given degree $d$. Suppose given a subset $S$ of the real solution set of the input polynomial system. The vanishing ideal
of $S$ denoted by $I(S)$ contains the real radical. By our approach, we can determine if $I(S)$ is contained in the real radical.
If it is, then $I(S)$ is the real radical. If not, then $S$ is not complete and a large $S$ is needed. See \cite{BHL:2016} for details
 of this approach.
We compare the performance of our techniques with the
popular SDP solver SeDuMi(CVX) which uses an interior point method.
On our illustrative examples, our approach has better accuracy, and the maximum rank condition can be guaranteed 
without misleading small eigenvalues. 

\section{Real radical and moment matrices}
\subsection{real radical}
Suppose that $x = (x_1, x_2, ... , x_n ) \in \mathbb{R}^n$
and consider a system of $m$ multivariate polynomials
$P = \{ p_1(x), p_2(x), ... , p_m(x) \}  \subseteq
\mathbb{R}[x_1, x_2, ... , x_n ]$ with
real coefficients.
Its solution set or variety is
\begin{equation} 
    V_\mathbb{R}  (p_1,...,p_m) = \{ x \in \mathbb{R}^n :  p_j(x) = 0,\; 1\leq j \leq m\}
\end{equation}

The ideal generated by $ P  =  \{   p_1,...,p_m  \}  \subseteq \mathbb{R}$  is:
\begin{equation}\label{RIdeal}
     \left\langle P  \right\rangle_\mathbb{R} = \left\langle  p_1,...,p_m  \right\rangle_\mathbb{R}  = \{ f_1 p_1 + ... + f_m p_m : f_j \in \mathbb{R}[x] ,    1\leq j \leq m  \}
\end{equation}
and its associated radical ideal over $\mathbb{R}$ is defined as
\begin{eqnarray}\label{RadRIdeal}
\sqrt[\mathbb{R}]{ \left\langle P \right\rangle }
              &=&
			 \{ f   \in \mathbb{R}[x] : f^{2t}+\Sigma_{j=1}^{s} q_j^{2} \in \left\langle P \right\rangle
								\;  \mbox{for some} \; q_j  \in \mathbb{R}[x], t \in \mathbb{N} \backslash \{0\} \} \label{defrrad}
\end{eqnarray}
A fundmental result \cite{BasuPollackRoy06} is:
\begin{theorem} \label{RealNull}
[Real Nullstellensatz]   For any ideal $I \subseteq \mathbb{R}[x]$ we have $\sqrt[\mathbb{R}]{ I } = I(V_R(I))$.
\end{theorem}
Consequently
\begin{eqnarray}\label{RadRIdeal1}
\sqrt[\mathbb{R}]{ \left\langle P \right\rangle }
              &=&  \{ f(x)   \in \mathbb{R}[x] : f(x) = 0
								\;  \mbox{ for all } \;  x \in V_\mathbb{R} (P) \} \;
\end{eqnarray}

 \begin{remark} \label{radical1}
 An ideal $I \subseteq \mathbb{R}[x]$ is real radical if and only if
     for all $ p_1, \cdots, p_m \in R[x]$:
 \begin{equation}
  p_1^2 + \cdots + p_m^2 \in I \Longrightarrow p_1, \cdots, p_m \in I.
 \end{equation}
 \end{remark}
For these and many other results see  \cite{BasuPollackRoy06} and the references
cited therein.
\subsection{Moment matrix}
\begin{defin}[Moment Matrix \cite{LaurentRostalski:12}]
\label{def:mmx1}
Given a linear form $\lambda \in \R[x]^*, x= (x_1\cdots x_n )$ which maps a polynomial to a real number. A symmetric matrix
\begin{equation}
M(\lambda) = (\lambda(x^{\alpha}x^{\beta}))_{\alpha,\beta \in \mathbb{N}^n}
\end{equation}
is called a moment matrix of $\lambda$ where $\mathbb{N} = \{0, 1, 2, \cdots \}$.
\end{defin}
Similarly, we define the truncated moment matrix.
\begin{defin} [Truncated Moment Matrix \cite{LaurentRostalski:12}]
\label{def:trmmx}
Given a linear form $\lambda_d \in (\R[x]_{2d})^*$, the truncated moment matrix of $\lambda_d$ is defined to be
\begin{equation}
M(\lambda_d) = (\lambda_d(x^{\alpha}x^{\beta}))_{\alpha,\beta \in \mathbb{N}_d^n}
\end{equation}
where $\mathbb{N}^n_d = \{\gamma \in \mathbb{N}^n: | \gamma | = \Sigma_{j=1}^n \gamma_j \leq d \}$.
\end{defin}
\begin{example} \label{exmtx4}
 Suppose $\lambda_1 \in \R[x,y]_{2d}^*$ for $d = 1$. Then
 \begin{equation}
 M(\lambda_1) = \begin{bmatrix}
 		u_{00} & u_{10} & u_{01} \\
 		u_{10} & u_{20} & u_{11} \\
 		u_{01} & u_{11} & u_{02}
 		 \end{bmatrix} 	 		
 \end{equation}
 Without loss, we  assume $u_{00} =1$ throughout this chapter.
\end{example}
The kernel of a positive semidefinite truncated moment matrix has the following ``real radical-like" property:
\begin{lemma} \cite{LaurentRostalski:12}
Assume $M(\lambda_d) \succeq 0$ and let $p, q_j \in \R[x]$, $f:= p^{2m} + \sum_jq_j^2$ with $m \in \mathbb{N}$,
$m \geq 1$. Then, $f \in \ker M(\lambda_d) \Rightarrow p \in \ker M(\lambda_d)$.
\end{lemma}
We also have the following therems which are known:
\begin{theorem}\cite[Lemma 3.1]{lasserre2008semidefinite} \label{thm:Momentmaxrank}
Suppose that the ideal $I = \langle f_1,\dots f_m \rangle_\mathbb{R}$
  with $\max_i(\deg(f_i)) =d$ and let $B$ be the coefficient matrix of
  $ \{ f_1,\dots f_m  \} \subseteq \mathbb{R}[x]$.
 Let $M(\lambda_d)$ be a truncated moment matrix such that $B \cdot M(\lambda_d) = 0$ and
 $M(\lambda_d) \succeq 0$.
 If the rank of $M(\lambda_d)$ is maximum then
 \begin{equation}
 \mathbb{P} \ker M(\lambda_d) \subseteq \sqrt[\R]{I}
 \end{equation}
\end{theorem}

\begin{theorem}(Flat extension theorem \cite{CurtoFialkow96})
\label{thm:flat}
Assume $M(\lambda_d) \succeq 0$. The following statements are equivalent:
\begin{description}
\item[(i)] There exists an extension $M(\lambda_{d+1}) \succeq 0$ and $\rank M(\lambda_{d}) =\rank M(\lambda_{d+1})$ \\
 \item[(ii)] $\ker M(\lambda_d)$ is \textit{ideal-like}.
 \end{description}
\end{theorem}
\begin{lemma}\cite[Theorem 3.4, Corollary 3.8]{lasserre2008semidefinite}
\label{thm:finite} 
Assume $M(\lambda) \succeq 0$ and $\rank M(\lambda_d) = \rank M(\lambda_{d-1}) =r$.
Then $J = \langle \mathbb{P} \ker M(\lambda_d)\rangle_\mathbb{R}$
is real radical and zero-dimensional.
One can extend $\lambda_d$ to $ \lambda =   \sum_{i=1}^r\alpha_i \lambda_{v_i} \in \R[x]^*$ where $\alpha_i >0$ and
$\{v_1,\dots,v_r\} = V_\mathbb{R} ( \mathbb{P} \ker M(\lambda_d))$.
Furthermore $\lambda = \lambda_d$ when $\lambda$ is restricted to $\R[x]_{2d}$.
\end{lemma}
\section{Computation of generators of the real radical up to a given degree}

Based on the maximum rank moment matrix, the geometric involutive form
\cite{reidwangwolk2015},
the results of Curto and Fialkow \cite{CurtoFialkow96}
and Lasserre et al. \cite{lasserre2008semidefinite}
we give an algorithm for computing the real radical up to a given degree $d$.

Throughout this section we consider a system of multivariate polynomials $\{f_1, \cdots, f_m\} \subseteq
\mathbb{R}[x_1, x_2, ... ,x_n]$ of degree $d = \max_i(\deg(f_i))$.
The associated real ideal is denoted
\begin{equation}
\label{def:RR-I}
I := \langle f_1, f_2, ... , f_m \rangle_\mathbb{R}
\end{equation}
and its associated real radical ideal is denoted by $\sqrt[\R]{I }$.

In particular we solve the following problem:
\begin{prob}\label{prob:rrtest1}
Given a system of polynomials $ \{f_1, \cdots, f_m\} \subseteq
\mathbb{R}[x_1, x_2, ... ,x_n]$ with associated ideal $I$ and an integer $d$ 
we give an algorithm to compute:
\begin{equation}
\label{def:truncatedRR}
\left( \sqrt[\R]{I }\right)_{(\leq d)} :=
  \{f \in \sqrt[\R]{ I }:  \deg(f) \leq d \}
 \end{equation}
\end{prob}
We will represent $\left( \sqrt[\R]{I }\right)_{(\leq d)}$ by polynomials corresponding to
vectors in $\ker M(\lambda_d)$ where $ M(\lambda_d)$ is the truncated moment
matrix to degree $d$ as defined in Definition \ref{def:trmmx}.

In order to obtain our main result we will require that
$\ker M(\lambda_d)$ is \textit{ideal-like} as defined by Curto and Fialkow \cite{CurtoFialkow96}.
We note that there is a bijective correspondence between vectors $v \in \ker M(\lambda_d)$ and polynomials given by $v \mapsto \mathbb{P} (v) =  v^T ( x^\alpha )_{\alpha \in \mathbb{N}^n}$
where $( x^\alpha )_{\alpha \in \mathbb{N}^n}$ is the vector of all monomials of degree $\leq d$
ordered in the same way as the rows of the moment matrix.  Conversely each polynomial $g$
 used to form the coefficient matrix $B$, is
mapped to a vector $\Vecc (g)$ in $\ker M(\lambda_d)$.

\begin{defi}[Ideal-Like truncated moment matrix \cite{CurtoFialkow96}]\label{defi:RG}
The kernel of a truncated moment matrix $M(\lambda_d)$ is \textit{ideal-like} of degree $d$ if the following two conditions are satisfied:
\begin{itemize}
\item If $f_1, f_2 \in \mathbb{P} \ker M(\lambda_d)$ then $f_1 + f_2  \in  \mathbb{P}\ker M(\lambda_d)$.
\item If $f \in \mathbb{P} \ker M(\lambda_d)$ and
 $g \in \mathbb{R}[x]$ has $\deg(fg) \leq d$, then $fg \in  \mathbb{P} \ker M(\lambda_d)$.
\end{itemize}
The ideal-like property is denoted as $RG$ in \cite{CurtoFialkow96}.
\end{defi}


 Our main result is:
\begin{theorem} \label{thm:membership1}
 Suppose that $I = \langle f_1,\dots f_m \rangle_\mathbb{R}$
  with $\max_i(\deg(f_i)) =d$ and let $B$ be the coefficient matrix of
  $ \{ f_1,\dots f_m  \} \subseteq \mathbb{R}[x]$.
 Let $M(\lambda_d)$ be a truncated moment matrix such that $B \cdot M(\lambda_d) = 0$ and
 $M(\lambda_d) \succeq 0$.
  If the rank of $M(\lambda_d)$  is maximum and $\ker M(\lambda_d)$ is \textit{ideal-like} then
 \begin{equation}
 \mathbb{P} \ker M(\lambda_d) = \left( \sqrt[\R]{I }\right)_{(\leq d)}
 \end{equation}
\end{theorem}
To prove the above theorem, we will need Theorem \ref{thm:Momentmaxrank},
Theorem \ref{thm:flat} and Lemma \ref{thm:finite}.

We now prove Theorem \ref{thm:membership1}.

\begin{proof}
 Suppose $\ker M(\lambda_d)$ is \textit{ideal-like}, $M(\lambda_d) \succeq 0$ and
 $M(\lambda_d)$ has maximum rank together with the other assumptions in
 Theorem \ref{thm:membership1}.

 Our goal is to show that
 $$\mathbb{P} \ker M(\lambda_d) = \left( \sqrt[\R]{I }\right)_{(\leq d)}.$$
 First by Theorem \ref{thm:Momentmaxrank}, the following direction is obvious:
 $$ \mathbb{P} \ker M(\lambda_d) \subseteq \left( \sqrt[\R]{I }\right)_{(\leq d)}. $$
 So we only need to show
 $$ \mathbb{P} \ker M(\lambda_d) \supseteq \left( \sqrt[\R]{I }\right)_{(\leq d)} $$
 By Theorems \ref{thm:flat} and \ref{thm:finite},
 $\lambda_d$ can be extended to $\lambda_{d+1}$ such that  $J = \langle  \mathbb{P} \ker M(\lambda_{d+1})\rangle_\mathbb{R}$ is real radical and
 zero-dimensional.
 Since $I \subseteq J$, we have $\sqrt[\R]{I} \subseteq J$.
 By Theorem \ref{thm:finite}, one can extend
$\lambda_d$ to $ \lambda = \sum_{i=1}^r\alpha_i \lambda_{v_i} \in \R[x]^*$  where
$\alpha_i > 0$ and
$\{v_1,\dots,v_r\} = V_\mathbb{R}( \mathbb{P} \ker M(\lambda_{d+1})) = V_\mathbb{R} (J)$ and $\lambda_{v_i}$
is an evaluation mapping at $v_i$ such that $\lambda_{v_i}(f) = f(v_i)$.
Thus $\lambda_d = \sum_{i=1}^r\alpha_i \lambda^{(d)}_{v_i}$ where
$\lambda^{(d)}_{v_i}$ is the truncated linear form of $\lambda_{v_i}$.
Since $\sqrt[\R]{I} \subseteq J$, we have
$\{v_1,\dots,v_r\} \subseteq V_\mathbb{R}(\sqrt[\R]{I} )$.

Now we can prove the other inclusion:
$$ \mathbb{P} \ker M(\lambda_d) \supseteq \left( \sqrt[\R]{I }\right)_{(\leq d)} $$
So we let $g \in \left( \sqrt[\R]{I }\right)_{(\leq d)}$ and we want to show that
$g \in \mathbb{P} \ker M(\lambda_d)$, that is to show that
$\Vecc (g)^T M(\lambda_{d}) = 0$.

Since $g \in \sqrt[\R]{I}$ with $\deg(g) \leq d$, we have $g(v_i) =0, i=1, \dots ,r$.
Therefore, we have $g^2(v_i) = \Vecc (g)^T M(\lambda^{(d)}_{v_i}) \Vecc(g) = 0$.
Since  $M(\lambda^{(d)}_{v_i} ) \succeq 0$ , we have $\Vecc (g)^T M(\lambda_{v_i})= 0$
for $ i= 1, \ldots ,  r$.
Hence $\sum_{i=1}^r \alpha_i \Vecc (g)^T M(\lambda^{(d)}_{v_i}) = 0$, so $\Vecc (g)^T M(\lambda_{d}) = 0$ and $ g \in \mathbb{P} \ker M(\lambda_d)$ which is what we wanted to show.
\qed
\end{proof}
By Theorem \ref{thm:membership1}, we now have a complete algorithm to Problem \ref{prob:rrtest1}

\begin{algorithm}[h]
\caption{RealRadical($F, d$)}
\label{alg:rrtest}
\SetKwData{Input}{Input}
\SetKwData{Output}{Output}
\Input{$F = \{f_1,\dots, f_m\} \subseteq \mathbb{R}[x]$, $x \in \mathbb{R}^n$, an integer $d \geq \deg(F)$.}\;
 Set $F'$  to the prolongation of $F$ to degree $d$ \\
\Repeat{$\dim F' = \dim F''$ }{
$B :=  \mbox{CoeffMtx}(F')$  \\
Solve for maximum rank moment matrix $M(\lambda_{d})$ such that $B^T M(\lambda_{d}) = 0, M(\lambda_{d}) \succeq 0$ by Algorithm \ref{alg:frprimal}.\\
$F'' := \mathbb{P} (\ker M(\Lambda_{d}))$ \\
Compute $\GIF(F'')$  \\
Project/ Prolong $\GIF(F'')$ to degree $d$:  $F' := \GIF(F'')_{(\leq d)} $.  \\

}
\Output{$F'$, a basis for $ \{ f \in \sqrt[\R]{I}: \deg(f) \leq d \}$}
\end{algorithm}
In Algorithm \ref{alg:rrtest}, $\mbox{CoeffMtx}$ computes the coefficients in the
monomial basis, although potentially other bases could be used.
It exploits the property that the the $\GIF$ algorithm obtains polynomials in a form that satisfies the ideal-like property.
In particular note that for a given $f$ in Definition \ref{defi:RG}, $fg = \sum_{\alpha} a_{\alpha}x^{\alpha}f$ is expanded in term of so-called prolongations by monomials $x^{\alpha}$. The invariance of geometric involutive bases under prolongation-projection implies that each $x^{\alpha}f$ is in the basis, and by superposition $fg$ is also in the basis. We note that Pommaret involutive bases don't necessarily satisfy the ideal-like property but can be extended easily by an explicit algorithm to such basis \cite{GerdtBlinkov:1998,Seiler2010}. Groebner bases can also
be extended, by essentially reformulating them as involutive basis \cite{GerdtBlinkov:1998}.

Involutivity originates in the geometry of differential equations.  See Kuranishi \cite{Kuranishi:1957} for a famous proof of termination of Cartan's prolongation algorithm
for nonlinear partial differential equations.
A by-product of these methods has been their implementation for linear
homogeneous partial differential equations with constant coefficients, and consequently for
polynomial algebraic systems.  See \cite{GerdtBlinkov:1998} for applications and symbolic algorithms for
polynomial systems.
The symbolic-numeric version of a \textdef{geometric involutive form, GIF}, was first described and implemented in
Wittkopf and Reid \cite{ReidWittkopf:2001}.  It was applied to approximate symmetries of differential
equations in \cite{BLRSZ:2004} and to polynomial solving in \cite{ReidZhi2009,ReidTangZhi:2003,SRWZ2010}.
See \cite{WuZhi12} where it is applied to the deflation of multiplicities in
multivariate polynomial solving.
For more details and examples see \cite{ReidWangWu:14,BLRSZ:2004}.
The details of the $\GIF$ algorithm, including, prolongations and projections, can be found in our earlier work \cite{reidwangwolk2015} and in chapter 2.

\section{SDP and facial reduction}
A symmetric matrix $M$ of sizes $k \times k$ is called positive semidefinite, denoted as $M \succeq 0$, if one of the following two criteria is satisfied:
\begin{enumerate}
\item $x^T M x \geq 0$ for all $x \in \R^k$.
\item All eigenvalues of $M$ are non-negative.
\end{enumerate}
Similarly, a symmetric matrix $M$ of sizes $k \times k$ is called positive definite, denoted as $M \succ 0$, if one of the following two criteria is satisfied:
\begin{enumerate}
\item $x^T M x > 0$ for all $x \in \R^k$.
\item All eigenvalues of $M$ are strictly positive.
\end{enumerate}
The set of all $k \times k$ symmetric matrices are denoted as $\Ss^k$.
The cone of $k \times k$ all positive semidefinite matrices is denoted as $\Ss_+^k$. The cone of $k \times k$ all positive definite matrices is denoted as $\Ss_{++}^k$.
\begin{defi}[Trace product]
Given two symmetric matrices $A,B$, we define the trace inner product $\langle A, B\rangle = \trace(A^TB) = \sum_{ij}A_{ij}B_{ij}$.
\end{defi}
\begin{defi}
 Suppose $A_1,...,A_l \in \R^{k\times k}$, the linear operator $\A$ from $\R^{k \times k}$ to $\R^l$ is defined as:
 \begin{equation}
  \A(X) = [\langle A_1, X \rangle, ..., \langle A_l, X \rangle]^T, X \in \R^{k\times k}
 \end{equation}
 The adjoint operator of $\A$ from $\R^l$ to $\R^{k\times k}$, denoted as $\A^*$, is defined as:
 \begin{equation}
  \A^*y = \sum_{i=1}^l A_i y_i, y \in \R^l
 \end{equation}
\end{defi}
\begin{defi} \label{defi:vec}
Given a matrix $H = (a_{ij})_{1 \leq i,j \leq k} \in \R^{k\times k}$, define $\Vecc(H)$ to be the vectorization of $H$, i.e.,
$$\Vecc(H) = [a_{11}, a_{12},\dots,a_{1k}, a_{21},a_{22},\dots,a_{k1},\dots,a_{kk}]^T$$
The matrix representation of the linear operator $\A$, denoted as $A$, is $A = [\Vecc(A_1), ..., \Vecc(A_l)]^T$.
\end{defi}

\subsection{Face, minimal face and facial structure}
We give a brief introduction to faces, minimal faces, and  lemmas about facial structure.
The definitions below can be found in
\cite{bw1,bw3,ScTuWonumeric:07,DrPaWo:14,MR3108446}.
\begin{defi}
	Given convex cones $F,K$ and $F\subseteq K$, we call $F$ a
	\textdef{face of $K$, $F\unlhd K$} if
	\[
		x,y \in K, x+y \in F \implies x,y \in F.
	\]
Given a nonempty covex subset $S$ of $K$, the \emph{minimal face} of $K$ containing $S$ is defined to be the intersection of
all faces of $K$ containing $S$.
\end{defi}

\begin{defi}
Suppose $F$ is a face of $\Ss_+^k$, the \emph{orthogonal complement} of $F$ denoted as $F^{\perp}$, is defined to be $F^{\perp} =\{ Z \in \Ss^k: Z \cdot X = 0, \forall X \in F \}$.
The dual cone of $F$, denoted as $F^*$, is defined to be $F^* =\{  Z \in \Ss^k: Z \cdot X \succeq 0, \forall X \in F\}$.
\end{defi}
The following lemmas about the facial structure of the semidefinite cone $\Ss_+^k$ are well known, see e.g. \cite{SaVaWo:97}.
\begin{lemma}
 Any face F of $\Ss_+^k$ is either $0$, $\Ss_+^k$ or
 \begin{equation}
  F = \{ X \in \Ss^k: X = U M U^T, M \in \Ss_+^r \}
 \end{equation}
 where $U$ is an $k \times r$ matrix.
\end{lemma}
\begin{lemma}
Suppose $F$ is a face of $\Ss_+^k$ and $W \in \Ss_+^k$. Then $\Ss_+^k \cap \{W\}^{\perp}$  and $F \cap \{W\}^{\perp}$ are faces of $\Ss_+^k$, where $\{W\}^{\perp} =\{X \in \Ss^k: X \cdot W = 0\}$.
\end{lemma}
\subsection{Facial reduction}
The idea of facial reduction was originally developed by Borwein and Wolkowicz \cite{bw1,bw3} in the 1980s.
However it has been nontrivial to develop practical algorithms implementing facial reduction. Only recently have practical algorithms been developed. For example it was recently applied to solve the large sensor network localization problems \cite{kriswolk:09,ChDrWo:14}.

We consider the set
 $F_P = \{X \in \Ss^k : \A(X) = b, X \succeq 0\}$ which is also the form of moment matrix SDP optimization problem considered
 in this thesis, clearly $F_P$ is a convex subset of $\Ss^k$. The following theorem gives information on the facial structure of $F_P$:
 \begin{theorem}[{\cite[SDP version of Lemma 28.4]{MR3108446}} ] \label{thm:FR}
 Define $F_{\min}$ to be the minimal face containing $F_P$. $\A^*$ is the adjoint of $\A$ defined before.
 For a face $F\unlhd \Ss_+^k$ containing $F_P$, the following holds :
 \begin{equation}
 \left\{\begin{array}{c}
  \hspace{-2cm}(I) \quad   \A( X)= b,  X \in F \\
  (II) \quad  b^Ty=0,\; Z= \A^*y \in F^* \backslash\; F^{\perp}
 \end{array}\right\}
 \Rightarrow X \in \{Z\}^{\perp} \cap F  \subset F.
 \end{equation}
 In addition, $F = F_{\min}$ if and only if $(II)$ has no solution.
 \end{theorem}

The matrix $Z$ is called the \emph{exposing vector} of $F$. Each time $(II)$ is solved, an exposing vector $Z$ is obtained and can be used to update $F \leftarrow \{Z\}^{\perp} \cap F$.
Repeating this process until $(II)$ is infeasible ($(II)$ admits no solution), we get a sequence of
faces containing $F_P$: $F_0 \supset F_1 \supset F_2 \supset \cdots  \supset F_{\min} \supset F_p$ where $F_0 = \Ss_+^k$ and $F_{i+1} = F_i \cap\{Z_i\}^{\perp}$. This iteration process to find the minimal face $F_{\min}$ is called \emph{facial reduction} on the primal form
and is guaranteed to terminate in at most $n - 1$ iterations \cite{MR2724357}. The minimal number of facial reductions is called
the \emph{singularity degree}.

The correctness of Theorem \ref{thm:FR} in the SDP case is due to the following theorem:
\begin{theorem}[Primal Theorem of alternative \cite{ScTuWonumeric:07,DrPaWo:14}]
\label{eq:thmalt}
Suppose $\A : \Skp \rightarrow \R^l$ is a linear transformation, $b \in
\R^l$, $P \in \Sk$ and $Z \in \Sk$. Then exactly one of the following
alternative systems is consistent:
\begin{subequations}
\begin{align}
 (I)	\quad & 0\prec P \in F:= \{P \in \Sk :  \A( P)= b,
P \succeq 0\} \quad \mbox{(Slater)}  \label{myeqna}\\
	&\hspace{1cm} \nonumber \\
 (II)	\quad & 0\neq Z \in D:= \{Z \in \Sk : Z = \A^*y \succeq 0,  b^Ty=0\}. \quad \mbox{(Auxiliary)} \label{myeqnb}
\end{align}
\end{subequations}
\end{theorem}
\begin{proof}
	Note that if (II) is consistent, then $Z$ exposes a face of
	$\Snp$ that contains the minimal face $ (F,\Snp)$. That is,~for $P\in F$ we have
	\[
		\trace ZP = \trace (\A^*y)P=y^Tb =0.
	\]
The remainder of the proof can be found in \cite{ScTuWonumeric:07,DrPaWo:14}. 
\end{proof}
Equation (\ref{myeqna}) is called the \emph{primal problem} and equation (\ref{myeqnb}) is called the \emph{auxiliary problem}.

\subsection{Facial reduction maximum rank algorithm}

Our facial reduction algorithm follows from Theorem \ref{thm:FR}. We use the following Lemmas to convert $(I), (II)$ of Theorem
\ref{thm:FR} to equivalent problems which are easier and more practical to solve. The proofs of these Lemmas can be found in the Appendix.
\begin{lemma} \label{lemma:1}
Suppose a face is given as $ F = \{ X \in \Ss^k: X = U M U^T, M \in \Ss_+^r \}$. Then 
\begin{eqnarray}
   \exists X \in F, \; \A( X)= b \iff \exists \bar X \in \Ss_+^r, \; U^T \A U (\bar X) = b, 
\end{eqnarray}
where $ U^T \A U$ is a linear operator from $\Ss^{r}$ to $\R^l$ defined as 
\begin{equation}
  U^T \A U(\bar X) = [\langle U^T A_1 U, X \rangle, ..., \langle U^T A_l U, X \rangle]^T, \bar X \in \Ss^{r}.
 \end{equation}
\end{lemma}
\begin{lemma}\label{lemma:2}
 Suppose $F = \{ X \in \Ss^k: X = U M U^T, M \in \Ss_+^r \}$. Then 
 \begin{eqnarray}
  & \exists Z= \sum_{i=1}^l A_i y \in F^* \backslash\; F^{\perp}, \;b^Ty=0 \label{eq:Expo} \\
  & \iff \nonumber \\
  & \exists \bar Z= \sum_{i=1}^l U^T A_i U y \succeq 0 \neq 0, \; b^Ty=0 \label{eq:expo}
 \end{eqnarray}
\end{lemma}

\begin{lemma}\label{lemma:3}
 Suppose $Z$ is an exposing vector satisfying (\ref{eq:Expo}) and $\bar Z$ satisfying (\ref{eq:expo}) with $V = \nul(\bar Z)$, $ F = \{ X \in \Ss^k: X = U M U^T, M \in \Ss_+^r \}$ is the face.
 Then 
 \begin{equation}
  \{Z\}^{\perp} \cap F = \{ X \in \Ss^k: X = UV \bar{M} V^T U^T, \bar{M} \in \Ss_+^{\bar r} \}
 \end{equation}

\end{lemma}

Recall in Algorithm \ref{alg:rrtest}, we need to find $M(\lambda_{d})$ such that $B^T M(\lambda_{d}) = 0, M(\lambda_{d}) \succeq 0$. All such moment matrices form a convex subset of $\R^{k \times k}$. Also in general, all the moment matrix $M(\lambda_{d})$ form an affine subspace $\A(X) = b$. The 
construction of $\A$ is described in \cite{reidwangwolk2015}. So the set $\{ M(\lambda_{d}): B^TM(\lambda_{d}) = 0, M(\lambda_{d}) \succeq 0 \}$ can be converted to a convex set $\F_{p} :=  \{X \in \Ss^k : \A(X) = b, B^TX = 0, X \succeq 0\}$.
The algorithm to use facial reduction to find maximum rank solutions of $\F_p$ in Algorithm \ref{alg:rrtest} is summarized as follows:

\begin{algorithm}[H]
\caption{
Facial reduction on the primal. Compute the minimal face $F_{min} := U \Ss_+^d U^T$ of $\Ss_+^k$ containing $\F_{p}$,
where $\F_{p} :=  \{X \in \Ss^k : \A(X) = b, B^TX = 0, X \succeq 0\}$. Obtain the maximum rank solution of $\F_{p}$.
}
\label{alg:frprimal}
\SetKwData{Input}{Input}
\SetKwData{Output}{Output}
\Input{$\mathcal{A}: \Ss^k \rightarrow \mathbb{R}^l, b \in \mathbb{R}^l$,$B \in \R^{k \times m}$, set $j  = 1$, $U = I$}\;
\Repeat{(\ref{eq:auxi}) only has zero solution}{
 If $j = 1$, set $Z = BB^T$ . \\
 If $j > 1$, \begin{equation}\label{eq:auxi}\tag{$\Diamond$}
 \begin{aligned}
 &\mbox{find}\quad Z \succeq 0 \\
 &\mbox{subject to} \; Z = \sum_{i=1}^{l} A_i y_i, b^Ty =0 : y \in \mathbb{R}^l  
  \end{aligned}
 \end{equation} 

Find a basis $V$ for  $\nul (Z)$.\\
Update $ \A $ by setting $ A_{i} \leftarrow V^T A_{i}V, i = 1 \dots l$. \\
 Update $U$ by setting $U \leftarrow U \cdot V$. \\
 $j = j+ 1$
}
Solve $\A(  P )= b,   P \succ 0$. Solution of $\F_p$ is $X := UPU^T$.\\
\Output{$X$ which is maximum rank solution}
\end{algorithm}
\begin{theorem}[Maximum rank] \label{thm:maxrank}
Algorithm~\ref{alg:frprimal} returns a maximum rank solution of $\F_p$.
\end{theorem}
\begin{proof}
At step j, 
when an exposing vector $Z \succeq 0$ is found ($Z = BB^T$ when $j =1$ or $Z$ satisfies (\ref{eq:auxi}) when $j >1$), we can reduce the problem to an equivalent smaller problem without loss of information by Lemma \ref{lemma:1}, \ref{lemma:2}, \ref{lemma:3} and Theorem \ref{thm:FR}.
When (\ref{eq:auxi}) only has zero solution, we have reduced the problem to a minimal face
with no further facial reductions can be done according to Theorem \ref{thm:FR} and all the feasible solutions of $\F_p$ has the form $X:= UPU^T$. By Theorem \ref{eq:thmalt} when (\ref{eq:auxi}) only has zero solution, there exists
$P \succ 0$ such that $\A(  P )= b,   P \succ 0$. As a result, $X:= UPU^T$ is the maximum rank solution of $\F_p$ if we can find $P$ which is positive definite. 
\qed
\end{proof}
\begin{remark}[Singularity degree]
 The minimal number of facial reduction steps is called \emph{singularity degree}. The examples in Section \ref{sec:EX} show that some examples with singularity
 more than 1 can be accurately solved by Facial reduction heuristics. For more details, see \cite{S98lmi,DrusLiWolk:14}.
\end{remark}
\section { Projection method}\label{sec:DR}
In Algorithm~\ref{alg:frprimal}, we need to solve two problems: the auxiliary problem to solve is (\ref{eq:auxilinear}) and the primal problem after
facial reduction to solve is $\A(P) = b, P \succ 0$.
Essentially, we need to find the intersection between an affine subspace (linear constraints) and a positive semidefinite cone.
We consider the \textdef{Douglas-Rachford reflection-projection} (DR) method which involves projections and reflections between two convex sets.
These two convex sets are the affine subspace and the positive semidefinite cone in our case. There are also other projection-based methods, such as method of
\textdef{alternating projection} \cite{MR2849884}. We prefer the DR method as it displays better convergence properties in our tests. Also, unlike the
alternating projection method, which is likely to converge to the boundary of cone, the DR method is likely to converge to the interior of the cone which 
is needed in Algorithm \ref{alg:frprimal} for solving $\A(P) = b, P \succ 0$.
\subsection{Projection to the positive semidefinite cone} \label{subsec:cone}
 Given $X \in \Ss^k$, denote $\PP_{\Ss_+^k}(X, r)$ as the projection of $X$ to $\Ss_+^k$ such that the projected matrix has rank $r$, we have the following well-known theorem:
\begin{theorem}[Eckart-Young \cite{EckartYoung39}]\label{thm:projrank}
 Suppose $X \in \Ss^k$, the projection of $\PP_{\Ss_+^k}(X, r)$ with $r \leq k$ is:
$\PP_{\Ss_+^k}(X,r)  = V \PP_{\Ss^k_+}(D, r) V^T$ and $X = V D V^T$ is the eigenvalue decomposition of $X$ and $D$ is a diagonal matrix with all the eigenvalues
of $X$. $\PP_{\Ss^k_+}(D, r)$ is obtained by keeping the first $r$ largest positive eigenvalues unchanged while setting all the other eigenvalues to zero.
\end{theorem}
\subsection{Projection to an affine subspace} \label{subsec:affine}
 Suppose an affine subspace
is given as follows:
\begin{equation}\label{eq:affine}
 \left\{ X \in \Ss^k, \; \A(X) = b  \right\}
\end{equation}
To project $X$ from $\Ss^k $ onto the affine subspace (\ref{eq:affine}), we have the following well-known theorem:
\begin{theorem}\cite{matrix2000}
Given a matrix $\bar X \in \Ss^k$, and $\A, b$ as in (\ref{eq:affine}). Let $A$ be the matrix representation of $\A$ as defined in Definition (\ref{defi:vec}) and $A^{\dagger}$ be the Moore-Penrose pseudoinverse of $A$, i.e., $A^{\dagger} = A^T(AA^T)^{-1}$.
\begin{equation}
 \begin{array}{rcl}
	&\mbox{Suppose}\; X^*:= \;\argmin \{|| X - \bar X ||: \A(X)  = b\} \\
        &\mbox{Then}\;X^*= \bar X + A^{\dagger}(b - A  \bar X).
 \end{array}
\end{equation}
We denote $X^* = \PP_{\A}(X)$.
\end{theorem}

\subsection{Transform of the auxiliary problem }
\label{sec:linconvert}
The auxiliary problem (\ref{eq:auxi}) can be solved by CVX or other SDP solvers, but in order to get higher accuracy, we use Douglas-Rachford iteration.
To do that, we need to reformulate the auxiliary problem (\ref{eq:auxi}).
First, it is easy to see problem (\ref{eq:auxi}) can be converted to the form:
\begin{eqnarray} \label{eq:DRAuxi}
&&\mbox{Find $y \in \R^l$}:  b^T y = 0,  A^T y-\Vecc(Z)= 0,\nonumber\\
&& \hspace{2.5cm} Z \succeq 0 ,  \trace(Z) = 1.
\end{eqnarray}
We add the trace constraint to make sure $Z \neq 0$. If \ref{eq:DRAuxi} is infeasible then (\ref{eq:auxi}) only has zero solution.

In addition, the following theorem shows how to transform problem (\ref{eq:DRAuxi}) into a simpler form that is suitable for applying the Douglas-Rachford method.
\begin{theorem}  
Suppose $A$ is the matrix representation of the linear operator $\A$ and $( A^T)^{\dagger}$ is the Moore-Penrose pseudoinverse of $ A^T$.
Let $L = [ b^T \cdot ( A^T)^{\dagger};I - A^T \cdot ( A^T)^{\dagger};\Vecc(I)]$ and $R = [0;0;1]$.  Then problem 
(\ref{eq:DRAuxi}) is equivalent to the following:
\begin{equation} \label{eq:auxilinear}
 \mbox{Find $Z \in \Ss^k$}: L \cdot \Vecc(Z) = R, Z \succeq 0,
\end{equation}
\end{theorem}
\begin{proof}
 Let's assume $\Vecc(Z) =  A^T y$, then we have $A^T(A^T)^{\dagger} \Vecc(Z) =  A^T(A^T)^{\dagger} A^T y  = A^T y= \Vecc(Z)$
 since $(A^T)^{\dagger} A^T = I$. Also $(A^T)^{\dagger} \Vecc(Z) = (A^T)^{\dagger} A^T y = y$.

It is easy to verify the other direction, by making the substitution $ y = (A^T)^{\dagger}\Vecc(Z)$. \qed
\end{proof}
By our experiments, we found this formulation has the best performance when coupled with the Douglas-Rachford methods. So we use (\ref{eq:auxilinear}) for solving problem (\ref{eq:auxi}) in Algorithm \ref{alg:frprimal}.
\subsection{Douglas-Rachford method}
In Sections \ref{subsec:cone} and \ref{subsec:affine} we showed how to project a matrix to a positive semidefinite cone and a affine subspace.
Briefly speaking, the DR methods first project a matrix $X$ to the positive semidefinite cone, then reflect it by multiplying the projected matrix by 2
and subtracting $X$ from it.  Similarly, the resulting matrix is projected and reflected over an affine subspace as well. Finally the average of the original matrix
and the reflected matrix is taken to update $X$ to $X_{new}$. More details can be found in \cite{MR0084194}.
(See also e.g.,~\cite{ArtachoBoTa:13,MR3149115}.) We apply Douglas-Rachford to solve both the primal problem and the auxiliary problem.
One step of the Douglas-Rachford method is the following:
\begin{eqnarray}
 \label{eq:dr}
\begin{aligned}
 &Y = 2\PP_{\Ss_+^k}(X,r) - X, \\
 &Z = 2 \PP_{\A}(Y) - Y, \\
 &X_{new} = (X + Z)/2.
 \end{aligned}
\end{eqnarray}
At each step, we calculate the residual $Res := \|\A(Y) - b\|$, which is the residual after projecting onto the positive semidefinite cone. If the residual is less
than the given tolerance, we stop and return $Y$. According to the basic theorem on the convergence of
the sequence,
\cite[Thm 3.3, Page 11]{MR3149115}, 
the residuals of the projections of the iterates on one of the sets
have to be used for the stopping criteria.
We use the residual after the projection onto the SDP cone since we want
our final matrix to be positive semidefinite.
\subsection{Choosing the appropriate rank for the projections}\label{sec:illcdn}
 In practice, some problems appear to
be very ill-conditioned. One example is the geometric polynomial in Section \ref{sec:EX}. Those examples have eigenvalue decomposition of
the solutions from  problem (\ref{eq:auxi}) with some eigenvalues that are very small compared to the others, and the DR iterations converge very slowly. This indicates
the rank $r$ used in the projection $\PP_{\Ss_+^k}(X, r)$ can not be maximum.

To deal with such problems, we would have to project the matrix to a good rank $r$ matrix as described in Theorem~\ref{thm:projrank} when applying the DR method to (\ref{eq:auxilinear}) for solving
problem (\ref{eq:auxi}). In other words, at each step of facial reduction,
we are not computing the smallest possible face. Instead, we try to find a bigger but much more accurate face. So we may need more facial reductions but we can obtain more
accurate results.

The strategy we used to get this good matrix is to look at the eigenvalues of $Z$ in (\ref{eq:auxilinear}). We drop the eigenvalues which are significantly
smaller than the other eigenvalues and $r$ is
chosen to be the number of eigenvalues which are well conditioned.  For example, if the eigenvalues are $0.7, 0.2, 0.00002, 0, 0, 0$, we will choose
$r = 2$ instead of $3$ or $6$. After this, we will resolve (\ref{eq:auxilinear}) with the updated $r$ to obtain a more accurate face.

\section{A special case for determining positive dimensional real radical}
\begin{figure}[h!]
\begin{center}
\includegraphics[width=0.4 \textwidth]{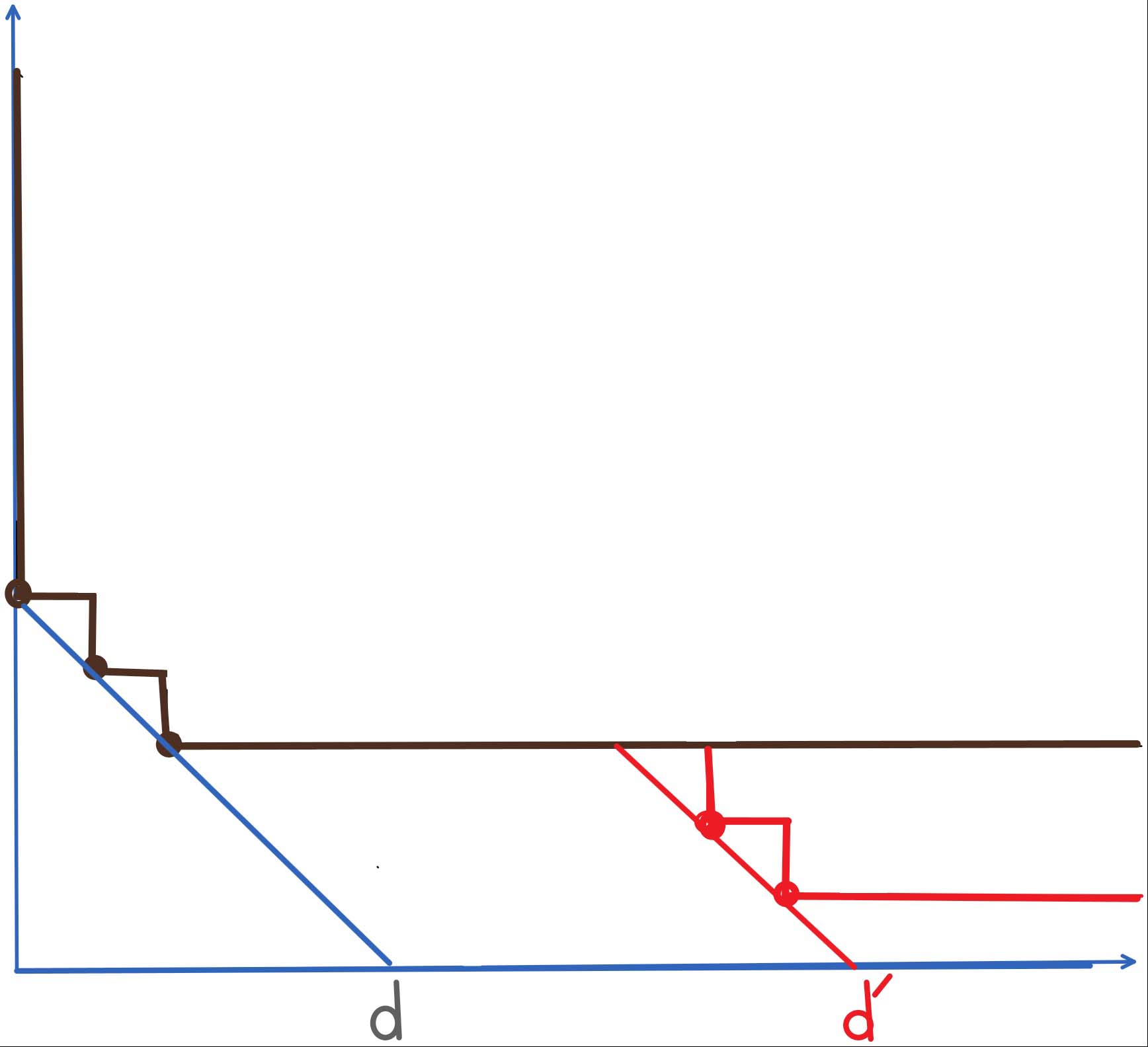}
\end{center}
\caption{In the Figure, the black monomial staircase represents the leading monomials of the generators of the real radical determined to degree $d$ by RealRadical($F, d$). The only way these can fail
to be a complete set of generators for the real radical is that there is a minimum degree
$d' > d$ where additional generators with leading monomials of exactly degree $d'$ shown in red are found outside black monomial staircase.
}
\label{Fig:Hilbert4}
\end{figure}

Our theorem on the determination of the real radical up to finite degree is illustrated graphically
in Figure \ref{Fig:Hilbert4}. Here suppose $F = \{f_1, ...,f_m\} \subset \R[x]$ and we applied Algorithm RealRadical($F, d$) for a given $d$, and that the resulting system has leading monomials shown as the corners of the black monomial staircase. See \cite{CLO1992} for the description of such diagrams. Then the system is prolonged and the kernel of its moment matrix is examined for new generators at degrees $d+1, d+2, \ldots $.
The only way that this is not a complete
generating set for the real radical (and that our conjecture fails), is that there is a minimum
degree $d' > d$ where after prolongation to $d'$ new generators are determined that lie outside
simple prolongations of the black leading generators. These have leading monomials shown in
red. Some times the completeness of the generating set at degree $d$ can be checked by
a critical point calculation. For example, if the critical point method shows that the variety
is real positive dimensional, then this could rule out the existence of the red staircase predicting
a $0$-dimensional real variety. In particular, if the number of red circles in Figure \ref{Fig:Hilbert4} is 1 and the variety of $F$ is real positive dimensional, then RealRadical($F, d$) returns the generators of
 $\sqrt[\R]{\langle F \rangle_{\R}}$. So we have the following theorem:
\begin{theorem}
  Given a system of polynomials $F = \{f_1, \cdots, f_m\} \subseteq
\mathbb{R}[x_1, x_2, ... ,x_n]$ with associated ideal $I$ and an integer $d$.   
Let $G = \{g_1, ..., g_k\} \subset \R[x]$ be the output of the RealRadical($F, d$) algorithm applied to $F$ and $s$ is the number of different polynomials of degree $d$ in $G$.
If $s = \binom{d + n-1}{n-1} -1$ and the variety of $F$ is real positive dimensional. Then
\begin{equation}
 \sqrt[\R]{ \langle F \rangle_{\R}} = \langle G \rangle_{\R}. 
 \end{equation}
\end{theorem} 
\begin{proof}
 By Theorem \ref{thm:membership1}, $\left( \sqrt[\R]{ \langle F \rangle_{\R}} \right)_{(\leq d)}= \spanl_{\R}G$. Suppose in contradiction $\sqrt[\R]{ \langle F \rangle_{\R}} \supset \langle G \rangle_{\R}$, 
 then there exists a $d^{\prime} > d$ such that $\left( \langle H \rangle_{\R}\right)_{(\leq d^{\prime})} \subset \left( \sqrt[\R]{ \langle F \rangle_{\R}} \right)_{(\leq d^{\prime})}$ where $H$ is the prolongation of $G$ to degree $d^{\prime}$.
  Therefore there exists a polynomial $\tilde{g} \in \spanl_{\R}\bar G$ but $g \notin \spanl_{\R} H$ with $\deg(\tilde{g}) = d^{\prime} > d$ where $\bar G = \{\bar g_1, ..., \bar g_{l} \}$ spans $\left( \sqrt[\R]{ \langle F \rangle_{\R}} \right)_{(\leq d^{\prime})}$.
    
  Now assume the number of different polynomials of degree $d^{\prime}$ in $H$ is $t$ and the number of different polynomials of degree $d^{\prime}$ in $\bar G$ is $\bar t$, then $t < \bar t$ because the existence of $\tilde{g}$.   From combinatorics, the number of different monomials of degree $d$ in $n$ variables is $\binom{d + n-1}{n-1}$.  Since $G$ is already involutive and $s =\binom{d + n-1}{n-1} -1$, we have $t = \binom{d^{\prime} + n-1}{n-1} -1$ as well. Also clearly $\bar t \leq \binom{d^{\prime} + n-1}{n-1}$, so we have $\bar t = \binom{d^{\prime} + n-1}{n-1}$ which means $\sqrt[\R]{ \langle F \rangle_{\R}}$ is a 0-dimensional real variety, a contradiction with the assumption that the variety of $F$ is real positive dimensional. So the theorem is proved.
  \QED
\end{proof}

\section{Examples}\label{sec:EX}
In this section, we give some examples. We used MATLAB version 2015a.
The computations were carried out on a desktop
with ubuntu 12.04 LTS, Intel Core\textsuperscript{TM}2 Quad CPU Q9550 @ 2.83 GHz $\times$ 4,
8GB RAM, 64-bit OS, x64-based processor.

We give the first examples (Ex.\ref{ex:eqn2} and Ex.\ref{ex:eqn3}) showing additional facial reductions for polynomials, that can be accurately approximated in practice. Our
previous attempts \cite{reidwangwolk2015} were not accurate.
\begin{example}[Reducible cubic]
\begin{equation}
\label{ex:eqn1}
 (x+y)(x^2+y^2+2)
\end{equation}
Note that the second factor has no real roots, so it is discarded and the real radical is generated by $(x + y)$.
The moment matrix corresponding to (\ref{ex:eqn1}) is a $10 \times 10$ matrix. The coefficient matrix $B$ is $[0, 2, 2, 0, 0, 0, 1, 1, 1, 1]^T$. Using Algorithm \ref{alg:rrtest}, after two facial reductions,
we obtained a maximum rank 4 moment matrix with residual less than $10^{-14}$ in less than 200 DR iterations and the generators of real radical is computed to degree 3. The GIF-FDR algorithm correctly yields to high accuracy the generator $(x + y)$ of the real radical to degree 1 as predicted by Theorem \ref{thm:GIFFDR}.

We compare it with SeDuMi(CVX), SeDuMi(CVX) obtains a rank 4 moment matrix with 9 decimal accuracy without maximizing the rank. However if we maximize the rank
(by maximizing the trace which is used in other examples as well) in CVX, the accuracy is only to 2 decimal places.
\end{example}
\begin{example}[Reducible quintic]
\begin{equation}
\label{ex:eqn2}
 (1+x+y)(x^4+y^4+2)
\end{equation}
The moment matrix corresponding to (\ref{ex:eqn2}) is a $21 \times 21$ matrix.
We solve this problem using Algorithm \ref{alg:rrtest}. Algorithm \ref{alg:rrtest} can get 14 decimal accuracy and a maximum rank
moment matrix of rank 6 in about 1300 DR iterations with 2 facial reductions. The output approximates the real radical ideal generated by $\langle 1+ x + y \rangle$ and its prolongations to degree 5. The GIF-FDR algorithm obtains the correct real radical generator $(1 + x + y)$ to degree 1 as predicted by Theorem \ref{thm:GIFFDR}.

We compare it with SeDuMi(CVX).  SeDuMi(CVX) can get a rank 6 moment matrix with 13 decimal accuracy without maximizing the rank. However if we maximize
the rank in CVX, we only get 9 decimal accuracy.
\end{example}
\begin{example}[Two variable geometric polynomial with 3 facial reductions]
\begin{equation}
\label{ex:eqn3}
1 + (x+y)+(x+y)^2 + (x+y)^{3}
\end{equation}
The moment matrix corresponding to (\ref{ex:eqn3}) is a $10 \times 10$ matrix. The coefficient matrix $B$ is $[2, 2, 2, 1, 0, 1, 1, 1, 1, 1]^T$.

This example is a demonstration of the ill-conditioned case discussed in Section \ref{sec:illcdn}. We first solve it using Algorithm~\ref{alg:frprimal} with rank $r$ to be
maximum in $\PP_{\Ss_+^k}(X,r)$, which returns solution
of rank 5 with residual $10^{-7}$ after 2 facial reductions. However, the DR method for solving the auxiliary problem (\ref{myeqnb}) converges very slowly.
So we check the eigenvalues
of solution of the auxiliary problem (\ref{myeqnb}). After the first facial reduction, the eigenvalues are $0.5, 0.2, 0.18,0.08, 0,0,0,0,0$. So we drop the fourth one and set $r = 3$.
We resolve (\ref{myeqnb}) using the DR method, which again is quite slow. So we check the eigenvalues and they are now $0.709,0.29,0.00002, 0, 0, 0, 0, 0, 0 ,0$. The third
one is very small so we drop it and set $r = 2$. Then we resolve (\ref{myeqnb}) with $r = 2$. This time the auxiliary problem is solved with residual $10^{-15}$.
Then a third facial reduction is done by setting $r = 3$ and the residual is $10^{-14}$.

After 3 facial reductions, the face is reduced to dimension $4$ and the moment matrix is obtained with residual
$10^{-13}$. The eigenvalues of the final moment matrix are $4.70, 3.48, 0.89, 0.59,\\ 0, 0, 0, 0 ,0 ,0$ which gives the correct maximum rank of $4$.

We compare it with SeDuMi(CVX) SDP solver. If we maximize the rank in CVX, we can obtain a moment matrix with residual about $10^{-9}$, the moment matrix has 8 positive
eigenvalues and the 5$th$ eigenvalue is $3 \times 10^{-5}$. So in order to get the correct maximum rank, the threshold has to be set to $10^{-4}$ which is not accurate.
If we do not maximize the rank, the residual is similar only the threshold is slightly better which is $10^{-5}$.

This example involves 3 facial reductions, the size of the problem after each facial reduction is $10, 9 ,7, 4$. Actually, this example has singularity degree 2 if we don't count the first ``trivial''
facial reduction. If we set the rank to be 5 when solving the auxiliary problem, it only returns a solution of rank 4 meaning
we can't reduce the problem to the minimal face by solving the auxiliary problem only once.  We tried the DR method to maximize the rank of the auxiliary problem with
random initial values 100 times, all yielding solutions of rank 4.

Actually we can prove the singularity is more than 1. We know the real radical of this polynomial system is  $\{1 + x + y, x + x^2 + xy, y + xy +y^2, x^2 + x^3 + x^2y,
xy + x^2y + xy^2,y^2 + xy^2 + y^3\}$ to degree 3. Let $N$ be the coefficient matrix of this polynomial system. Then $Q = V^TNN^TV$ will be the orthogonal complement of the primal problem
$\bar \A(X) = \bar b, X \succeq 0$ with rank 5 where $V^TB = 0$. If the singularity degree  is 1, then $Q = \sum_{i=1}^m\bar A_i y_i$ must be consistent
($\bar b^Ty =0 \implies y_0 =0$). By checking the rank of $[\bar\A,\svec(Q)]$ and $\bar \A$, we found the linear system is inconsistent so the singularity degree is 2.

Application of Algorithm \ref{alg:rrtest} yields the correct generators of the real radical  up to degree 3.  Application of GIF-FDR algorithm yields the generators of real radical to degree 1 which is $1 + x +y$.
\end{example}
\begin{example}\cite{BHL:2016} \label{ex4}
\begin{equation}
f = \{2yz-y,\;2y^2+y,\;xy,\;4x^2z+4z^3 + y\}
\end{equation}
The real radical of this polynomial system is \cite{BHL:2016}:
$$\{z^2+y/2, yz-y/2, y^2 + y/2, xz, xy, y+z\}$$
The moment matrix of this problem is $20 \times 20$. We use Algorithm~\ref{alg:frprimal} to solve for maximum rank moment matrix. The sizes of the SDP problem are [20, 16, 14, 8] after 3 facial reductions.
The residual of the auxiliary problem at each facial reduction is $10^{-15}, 10^{-14}$. (The first facial reduction is done by Matlab eigenvalue decomposition so we don't put its residual here.)
The moment matrix is solved with residual $10^{-13}$ and the maximum rank is 8.

We compare it with SeDuMi(CVX) which shows very poor performance. If we maximize the rank in CVX, the residual of the moment matrix solved by SeDuMi(CVX) is $8.5 \times 10^{-11}$ with
9 positive eigenvalues, of which 6 eigenvalues are greater than 0.1 and the other three eigenvalues are around $5 \times 10^{-7}$.  If we do not maximize the rank in CVX, then the residual is $8 \times 10^{-10}$.
But to get the correct rank, the threshold for the eigenvalues has to be set to $1\times 10^{-7}$.
So in general, it is very difficult to use SeDuMi(CVX) to get the correct maximum rank.
\end{example}

\begin{table}
\centering
\small{
\setlength{\tabcolsep}{.56667em}
\begin{tabular}{|l|c|c|c|c|c|c|}
\hline
  & min \# FR & max \# FR& rank (FR) & Singlty deg  & Res(FR) & Res(CVX)\\
\hline
  Ex \ref{ex:eqn1} & 2 & 3&  10, 9, 4 & 1 & $10^{-14}$ & $10^{-9}$ \\
  Ex \ref{ex:eqn2} & 2 & unknown &  21, 20, 6 & 1 & $10^{-14}$ & $10^{-9}$ \\
  Ex \ref{ex:eqn3} & 3 & 4 &  10, 9, 7, 4 & 2 & $10^{-13}$ & $10^{-9}$ \\
  Ex \ref{ex4} & 3 & 4 & 20, 16, 14, 8 & 2 & $10^{-13}$ & $10^{-9}$ \\
  \hline
\end{tabular}
}
\caption{\small{{\bf Comparison between facial reduction and SeDuMi (1)  }
All data is obtained by using minimal number of facial reductions; Here:
 min (max) \# FR means minimal (maximum) number of facial reductions in our tests; rank(FR) means
 the size of the problem after each facial reduction, the first one is the size of the original problem;
 Singlty degree is the singularity degree of the SDP problem after the 1st facial reduction; Res(FR) is the residual
 of the final moment matrix using facial reduction and DR iterations  (Algorithm \ref{alg:frprimal});
 Res(CVX) is the residual of the final moment matrix using CVX(SeDuMi).
 }}
 \label{table1}
\end{table}
\begin{table}
\centering
\small{
\setlength{\tabcolsep}{.56667em}
\begin{tabular}{|l|c|c|c|c|c|c|}
\hline
  & max rank & res each FR & \# DR each FR & thres FR & thres CVX \\
\hline
  Ex \ref{ex:eqn1} & 4 & $10^{-15},10^{-15 }$& 120, 7 & $10^{-16}$ & $10^{-12}$  \\
  Ex \ref{ex:eqn2} & 6 &  $10^{-15},10^{-14}$ &  267, 6 & $10^{-16}$ & $10^{-9}$  \\
  Ex \ref{ex:eqn3} & 4 &  $10^{-15},10^{-14},10^{-15}$ & 260, 143, 1& $10^{-16}$ & $10^{-5}$  \\
  Ex \ref{ex4} & 8 & $10^{-15}, 10^{-14},10^{-14}$ & 625, 192, 29 & $10^{-16}$ & $10^{-7}$ \\
  \hline
\end{tabular}
}
\caption{\small{{\bf Comparison between facial reduction and SeDuMi (2)}
All data obtained here is by using minimal number of facial reductions;
 max rank is the maximum rank of the moment matrix; res each FR is the residual of solving the corresponding SDP problem by DR after each facial reduction;
 \# DR each FR is the number of DR iterations to solve the corresponding SDP problem after each facial reduction;
 thres FR is the tolerance to obtain the correct maximum rank using facial reductions (Algorithm \ref{alg:frprimal});
 thres CVX is the tolerance to obtain the correct maximum rank using CVX(SeDuMi);
 }}
 \label{table2}
\end{table}
As the computations in the above examples and Table \ref{table1},\ref{table2} demonstrate, the traditional interior point SDP solver SeDuMi(CVX) is not the right choice for computing the maximum rank moment matrices
as it usually yields poorer performance when it is trying to maximize rank. It even gets better performance without maximizing the rank! With facial reductions and the DR
method, we can get much better accuracy and also the correct maximum rank.

In the above examples, Algorithm \ref{alg:rrtest} and GIF-FDR follow the same path except that GIF-FDR executes an extra step which reduces the degree of the output. Generally, however, the paths of these two algorithms can be quite different.

\section{Conclusion}\label{sect:Conclusion4}
SDP feasibility problems typically involve the intersection of the
convex cone of semi-definite matrices with a linear manifold.
Their importance in applications has led to the development of many specific algorithms.
However these feasibility problems are often marginally infeasible,
i.e.,~they do not satisfy strict feasibility as is the case for our polynomial applications.
Such problems are \emph{ill-posed} and \emph{ill-conditioned}.

This chapter is part of a series in which we exploit facial reduction and its application
 systems of real polynomial and differential equations for real solutions.
 The current work is directed at guaranteeing the maximal rank property and the ideal-like condition to ensure all the generators of the real radical up to a given degree are captured. It also establishes the first examples of additional
 facial reduction that are effective in practice for polynomial systems.

This builds on our work in \cite{reidwangwolk2015} in which we introduced facial reduction, for the class of SDP problems arising
from analysis and solution of systems of real polynomial equations for real solutions.
Facial reduction yields an equivalent smaller problem for which there are strictly feasible generic points.
Facial reduction also reduces the size of the moment matrices occurring in the application
of SDP methods.  For example the determination of a $k \times k$ moment matrix for a problem
with $m$ linearly independent constraints is reduced to a $(k-) \times (k-m)$ moment matrix by one facial reduction.
The high accuracy required by facial reduction and also the ill-conditioning commonly encountered in numerical
polynomial algebra \cite{Stetter:2004} motivated us to implement Douglas-Rachford iteration
in \cite{reidwangwolk2015}.

A fundamental open problem is to generalize the work of \cite{LasserreLaurentRostalski09,MR2830310} to positive dimensional ideals.
The algorithm of \cite{MWZ:2012,MA12} for a given input real polynomial system $P$, modulo the successful
application of SDP methods at each of its steps, computes a Pommaret basis $Q$:
\begin{equation}
\label{eq:RealRadConj4}
\sqrt[\R]{ \left\langle P
\right\rangle_\R  }  \; \;   \supseteq  \; \;  \left\langle Q
\right\rangle_\R  \; \;  \supseteq  \; \;  \left\langle P
\right\rangle_\R
\end{equation}
and would provide a solution to this open problem if it is proved that $ \left\langle Q
\right\rangle_\R = \sqrt[\R]{ \left\langle P
\right\rangle_\R  }$.
We believe that the work \cite{MWZ:2012,MA12} establishes an important
feature -- involutivity -- that will necessarily be a
main condition of any theorem and algorithm characterizing the real radical.
Involutivity is a natural condition, since any solution of the above open problem using SDP, if it establishes radical
ideal membership, will necessarily need (at least implicitly) a real radical Gr\"obner basis.
Our algorithm, uses geometric involutivity, and similarly gives an intermediate ideal, which constitutes another
variation on this family of conjectures.

An important open problem is the following:
\textit{Give an numerical algorithm, capable in principle
of determining an approximate real point on each component of a real variety.}
We note that the methods of Wu and Reid \cite{WuReid13} and Hauenstein \cite{Hauenstein12}
only answer this question under certain conditions, say that the ideal is real radical and defined by
a regular sequence.  Also see \cite{MR2389242}, which gives an alternative extension of complex numerical algebraic geometry to the reals, in the complex curve case.

Recently, Hauenstein et al \cite{BHL:2016} have made progress on this problem by using sample points determined by Hauenstein's critical point algorithm which is able to certify the generators of the real radical ideal in some cases. Our results Theorem \ref{thm:membership1} enables the determination of the generators up to a given degree. Thus gives an answer to the open problem of real radical ideal membership test  left in \cite{BHL:2016}. Potentially, the efficiency for computing the sample points can also be improved which will be described in a subsequent work.

\cleardoublepage
\addcontentsline{toc}{section}{Index}
\label{ind:index}
\printindex

\addcontentsline{toc}{section}{Bibliography}

\bibliography{.master,.edm,.psd,.bjorBOOK,.GI,.realAlg}

\appendix
\section{Proofs of Lemma \ref{lemma:1}, \ref{lemma:2}, \ref{lemma:3}.}
\subsection{Proof of Lemma \ref{lemma:1}}
First suppose there exists $X = U M U^T $ satisfying  $\A( X)= b$ , then we have $U^T \A U (M) = \A (U M U^T) = b$ due to the cyclic property of the trace product.

For the other direction, suppose there exists $\bar X $ satisfying  $U^T \A U ( \bar X)= b$, let $X = U \bar X^T U^T$ then it is easy to see $\A(X) = b$ as well.
\qed

\subsection{Proof of Lemma \ref{lemma:2}}
Suppose \eqref{eq:Expo} holds, there exists $Z= \sum_{i=1}^l A_i y \in F^*$ which means $\langle Z, U M U^T \rangle \succeq 0$ for all $M \in \Ss_+^r$ and $\langle U^T Z U, M \rangle \succeq 0$  for all $M \in \Ss_+^r$. Also
$Z \notin F^{\perp}$ which means $\langle U^T Z U, M \rangle \neq 0$ for some $M \in \Ss_+^r$ which indicates $ U^T Z U \neq 0$.

Now suppose \eqref{eq:expo} holds, since $\bar Z = U^T Z U \succeq 0$, we have $\langle  Z , U M U^T \rangle = \langle U^T Z U, M \rangle \succeq 0$ for all $M \in \Ss_+^r$. Hence $Z \in F^*$. Since $\bar Z \neq 0$, we have $Z \notin \nul(U^T)$ so $Z \notin F^{\perp}$.
\qed
\subsection{Proof of Lemma \ref{lemma:3}}
First, suppose $X = UV \bar{M} V^T U^T$, then $\langle Z, X\rangle = \langle U^T Z U, V \bar{M}V^T \rangle = 0$ which means $Z X =0$ since $Z\succeq 0, X\succeq 0$. So $X \in \{Z\}^{\perp}$ and $X \in F$.

For the other direction, if $X \in F$, then $X = U M U^T$ for some $M \in \Ss_+^r$. If $X \in \{Z\}^{\perp}$, then $XZ =0$ which means $\langle X, Z \rangle = \langle M, U^T Z U\rangle  = \langle M, \bar Z\rangle =0 \Rightarrow M \bar Z = 0$. Hence $M = V \bar M V^T$ for $V = \nul(\bar Z)$ and $X = U V \bar M V^T U$ for some $\bar M \in \Ss_+^{\bar r}$.
\qed

\end{document}